\documentclass[12pt,twoside]{amsart}

\usepackage{hyperref}
\usepackage{amsthm, amsmath, amscd, amssymb,centernot}
\usepackage[all]{xy}
\usepackage[T1]{fontenc}
\usepackage[left=2.5cm,top=2.5cm,bottom=3cm,right=2.5cm]{geometry}
\newcommand{\ints}{\mathbb Z}

\newcommand{\str}{\mathcal{O}}

\newcommand{\proj}{\mathbb{P}}

\newcommand{\complex}{\mathbb C}

\theoremstyle{plain}
\numberwithin{equation}{section}
\newtheorem{theorem}{Theorem}[section]
\newtheorem*{theorem*}{Theorem}

\newtheorem{lemma}[theorem]{Lemma}
\newtheorem{corollary}[theorem]{Corollary}

\newtheorem*{conjecture*}{Nagata Conjecture}
\newtheorem*{conjecture1*}{SHGH Conjecture}
\theoremstyle{definition}

\newtheorem{remark}[theorem]{Remark}
\newtheorem{example}[theorem]{Example}

\begin{document}
\title{Multi-point Seshadri constants on ruled surfaces}
\author[Krishna Hanumanthu]{Krishna Hanumanthu}
\address{Chennai Mathematical Institute, H1 SIPCOT IT Park, Siruseri, Kelambakkam 603103, India}
\email{krishna@cmi.ac.in}

\author[Alapan Mukhopadhyay]{Alapan Mukhopadhyay}
\address{Chennai Mathematical Institute, H1 SIPCOT IT Park, Siruseri, Kelambakkam 603103, India}
\email{alapan@cmi.ac.in}

\subjclass[2010]{Primary 14C20; Secondary 14H50}
\thanks{Authors were partially supported by a grant from Infosys Foundation}

\date{January 20, 2017}
\maketitle

\begin{abstract}
Let $X$ be a surface and let $L$ be an ample line bundle on $X$. We
first obtain a lower bound for the Seshadri constant $\varepsilon(X,L,r)$, when
$r \ge 2$. We then assume that $X$ is a ruled surface and study
Seshadri constants on $X$ in greater detail. We also make
precise computations of Seshadri constants on $X$ in some cases.
\end{abstract}

\section{Introduction}
Seshadri constants have been defined by Demailly \cite{Dem} as a
measure of local positivity of a line bundle on a projective
variety. The foundational idea is a criterion for
ampleness given by Seshadri \cite[Theorem 7.1]{Har70}. For a detailed
introduction and the state of current research on
Seshadri constants, see \cite{primer}. 

Let $X$ be a smooth complex projective variety and let $L$ be a nef line bundle on
$X$. Let $r \ge 1$ be an integer and let $x_1,\ldots,x_r$ be distinct
points of $X$.
The {\it Seshadri constant} of $L$ at $x_1,\ldots,x_r \in X$ is defined as: 
$$\varepsilon(X,L,x_1,\ldots,x_r):=  \inf\limits_{\substack{C \subset
    X{\rm ~a~ curve~ with} \\C \cap \{x_1,\ldots,x_r\}
  \ne \emptyset}} \frac{L\cdot C}{\sum\limits_{i=1}^r {\rm
    mult}_{x_i}C}.$$ 

It is easy to see that the infimum above is the same as the infimum taken over
irreducible, reduced curves $C$ such that $C \cap \{x_1,\ldots,x_r\}
\ne \emptyset$.

The Seshadri criterion for ampleness says that a line bundle $L$ on a
smooth projective variety is ample if and only if $\varepsilon(X,L,x)>0$ for
every $x \in X$. 

Now define $$\varepsilon(X,L,r) : = \max\limits_{x_1,\ldots,x_r \in X}
\varepsilon(X,L,x_1,\ldots,x_r).$$ 

It is known that $\varepsilon(X,L,r)$ is attained at a {\it very 
general} set of points $x_1,\ldots,x_r \in X$; see
\cite{Ogu}. This means that $\varepsilon(X,L,r) = 
\varepsilon(X,L,x_1,\ldots,x_r)$ for all 
 $(x_1,\ldots,x_r)$ 
outside some countable union of proper Zariski closed sets in 
$X^r  = X \times X \times \ldots \times X$.

The following is a  well-known upper bound for Seshadri
constants. 
Let $n$ be the dimension of $X$. Then for
any $x_1,\ldots,x_r \in X$, 
\begin{eqnarray*}\label{wellknown}
\varepsilon(X,L,x_1,\ldots,x_r) \le
\sqrt[n]{\frac{L^n}{r}}.
\end{eqnarray*}

In view of this upper bound, a lot of research is aimed at finding
good lower bounds for the Seshadri constants of ample line bundles,
primarily when $X$ is a surface. 
There has been extensive work on computing or finding lower bounds for
Seshadri contants on surfaces, mainly 
in the single point case ($r=1$). The
multi-point case ($r \ge 2$) is also of interest and there are results in
this case in various situations. 
Some results for multi-point Seshadri constants can be found in
\cite{Bau,Far,FSST,Han1,Har,HR08,SS,Sze,Sze08}.

Let $X$ be a surface and let $L$ be an ample line bundle on $X$. 
One of the crucial ideas in finding lower bounds is the observation
that if a Seshadri constant $\varepsilon(X,L,x_1,\ldots,x_r)$ is {\it
  sub-maximal} (i.e., $\varepsilon(X,L,x_1,\ldots,x_r) < \sqrt{L^2/r}$),
then there is actually an irreducible and reduced curve $C$ which
passes through at least one of the points $x_i$ such that 
$\varepsilon(X,L,x_1,\ldots,x_r) = \frac{L\cdot C}{\sum_i \text
  {mult}_{x_i}(C)}$.
Such curves are called {\it Seshadri curves}.  See 
\cite[Proposition 1.1]{BS} for a proof of their existence for
sub-maximal single-point Seshadri constants which generalizes 
easily to the multi-point case.

In our first main result,  Theorem \ref{main-general}, we consider 
an arbitrary smooth projective surface $X$ and an ample line bundle $L$ on
$X$, and show that, for an integer $r \ge 2$, the Seshadri constant satisfies 
$\varepsilon(X,L,r) \ge \sqrt{\frac{r+2}{r+3}}\sqrt{\frac{L^2}{r}}$, or
there is a Seshadri curve $C$ on $X$ passing through $s \le r$ very general
points with multiplicity one at each point. This bound is a generalization of 
\cite[Theorem 2.1]{Han1}, where it was proved for surfaces with Picard 
number 1. 

In some situations, it is possible to either rule out the existence of a
Seshadri curve passing through very general points with 
multiplicity one each, or limit the possibilities for such curves. 
We illustrate this phenomenon in a few cases. 
See Example \ref{k3} and Lemma \ref{rational-curve}.

In Section \ref{ruled-surfaces}, we 
study multi-point Seshadri constants on ruled surfaces. 
Motivated by the results of \cite{Fue,Syz}, where 
single-point Seshadri constants are studied on ruled surfaces,  
we make precise
computations of  multi-point Seshadri constants (Theorem
\ref{r<e}) when the number of points $r$ is small. Then we study the
case of rational ruled surfaces in greater detail. We
show in Theorem \ref{rational-ruled-seshadri} that, given a 
rational ruled surface $X$ and an ample line bundle
$L$ on $X$, there exists a large enough $r$ (depending on $L$) such that
there are no Seshadri curves for $L$ passing through $r$ very general points
with multiplicity one each.

Let $(X,L)$ be a polarized surface (that is, $X$ is a surface and $L$ is an ample line
bundle on $X$).  The {\it Nagata-Biran-Szemberg Conjecture}  
predicts that for large enough $r$, 
the Seshadri constant $\varepsilon(X,L,r)$ is maximal; i.e, 
$\varepsilon(X,L,r) = \sqrt{\frac{L^2}{r}}$. In fact, the conjecture
makes a precise prediction about how large $r$ should be. It says that the 
maximality holds for $r \ge k_0^2L^2$, where $k_0$ is the
smallest integer such that the linear system $|k_0L|$ contains
non-rational curves. Note that $k_0=3$ when $(X,L) =
(\proj^2,\str(1))$. In this sense, the Nagata-Biran-Szemberg
Conjecture is a generalization of the celebrated {\it Nagata
Conjecture} (see \cite[Remark 5.1.24]{L}, \cite{SS}, or \cite{Sze01} for more details). 
The Nagata-Biran-Szemberg Conjecture has
been verified asymptotically in \cite[Corollary 4.4]{SS}. 

The results in this paper provide further support for the 
Nagata-Biran-Szemberg conjecture. For small $r$, specifically for 
$r < k_0^2L^2$, the Seshadri constant may be small, but not for 
$r\ge k_0^2L^2$. This is evident from our results on ruled
surfaces. In Theorem \ref{r<e}, we show that the Seshadri constants 
are small on a ruled surface when $r$ is small. On a rational 
ruled surface, Theorem \ref{rational-ruled-seshadri} shows that  
$\varepsilon(X,L,r)$ asymptotically approaches the maximal value as 
$r$ increases.

%In Section \ref{general-result}, we prove a bound on Seshadri
%constants for arbitrary surfaces. In Section \ref{ruled-surfaces}, we
%study ruled surfaces. 

%{\bf Notation:} 
All the varieties we consider are defined over $\complex$, the field
of complex numbers.  A {\it surface} is a
two-dimensional smooth complex projective variety. 
%When we say $x_1,\ldots,x_r \in $ are {\it very
 % general} points, 
%we mean that $(x_1,\ldots,x_r)$ is
%outside a countable union of proper Zariski closed sets in 
%$X^r  = X \times X \times \ldots \times X$. 

%We will often simply write
%{\it general points} to refer to very general points in the above
%sense. 

{\bf Acknowledgements:} We sincerely thank Brian Harbourne, 
D. S. Nagaraj and Tomasz Szemberg for useful discussions. We also
thank the referee for pointing out a way to strengthen Theorem 2.1 
and numerous other suggestions which improved the exposition.

\section{A general lower bound}\label{general-result}
In this section, we prove a general lower bound for multi-point
Seshadri constants on any surface.  This is a 
generalization of \cite[Theorem 2.1]{Han1} 
where the same bound is proved for surfaces with Picard
number 1.

\begin{theorem}\label{main-general}
Let $X$ be a surface and let $L$ be an ample line bundle on $X$. Let
$r \ge 2$ be an integer. If $\varepsilon(X,L,r) <
\sqrt{\frac{r+2}{r+3}}\sqrt{\frac{L^2}{r}}$, 
then 
$\varepsilon(X,L,r) = \frac{L\cdot C}{s}$, where  
$C$ is an irreducible and reduced curve
on $X$ which passes through $s\le r$ very general points with multiplicity one
at each point.

Moreover, for every $r \ge 2$, there exists a polarized
surface $(X,L)$ such that $\varepsilon(X,L,r) <
\sqrt{\frac{r+2}{r+3}}\sqrt{\frac{L^2}{r}}$.
\end{theorem}
\begin{proof}

Suppose that $\varepsilon(X,L,r) <
\sqrt{\frac{r+2}{r+3}}\sqrt{\frac{L^2}{r}}$. In particular, the Seshadri
constant $\varepsilon(X,L,r)$ is not maximal. Then, as we noted in the introduction,
there exists an
irreducible and reduced curve $C$ such that 
$\varepsilon(X,L,r) = \frac{L\cdot C}{\sum_i m_i}$, where
$m_1,\ldots,m_r$ are the multiplicities of $C$ at some very general points
$x_1,\ldots,x_r$. 

Arrange the multiplicities in decreasing order, so
that $m_1 \ge m_2\ge \ldots \ge m_s > 0$ and $m_{s+1} = \ldots = m_r =
0$ for some $1 \le s \le r$. 

Since the points $x_1,\ldots,x_r$ are very general, there is a
non-trivial one-parameter family of irreducible and reduced curves 
$\{C_t\}_{t\in T}$ parametrized by some smooth curve $T$ and containing
points $x_{1,t},\ldots,x_{r,t} \in C_t$ with mult$_{x_{i,t}}(C_t) \ge
  m_i$ for all $1\le i \le r$ and $t \in T$. 

By a result of Ein-Lazarsfeld \cite{EL} and Xu \cite[Lemma 1]{X1}, we have 
\begin{eqnarray}\label{el}
C^2 \ge
m_1^2+\ldots+m_s^2-m_s.
\end{eqnarray}

First, suppose that $m_1=1$. Hence $m_1=\ldots=m_s=1$. In this case, the Seshadri constant
$\varepsilon(X,L,r)$ is computed by a curve $C$ which passes through $s
\le r$ very general points with multiplicity one each. 

Now we assume that $m_1 \ge 2$ and show that 
$\varepsilon(X,L,r) \ge
\sqrt{\frac{r+2}{r+3}}\sqrt{\frac{L^2}{r}}$, contradicting our
assumption.

We will apply the following special case of 
\cite[Lemma 2.3]{Han}, which holds when either $m_1 \ge 2$ and $s \ge 3$ or if $s=2$ then
$(m_1,m_2) \ne (2,2)$.
\begin{eqnarray}\label{ineq}
\frac{(s+3)s}{s+2}\left(\sum_{i=1}^s m_i^2 -m_s\right) \ge
  \left(\sum_{i=1}^s m_i\right)^2.
\end{eqnarray}

First, suppose that the conditions required for \eqref{ineq}
hold. That is:  $m_1 \ge 2$, $s \ge 3$ or if $s=2$ then
$(m_1,m_2) \ne (2,2)$. 

We have $(L\cdot C)^2 \ge L^2 C^2$, by the Hodge Index
Theorem. Hence the inequalities \eqref{el} and \eqref{ineq} give 
$$L\cdot C \ge \sqrt{L^2}\left(\sum_{i=1}^s m_i\right)
\sqrt{\frac{s+2}{s(s+3)}} \ge \left(\sum_{i=1}^s m_i\right) \sqrt{\frac{L^2}{r}}\sqrt{\frac{r+2}{r+3}}.$$
Thus $$\varepsilon(X,L,r) = \frac{L\cdot C}{\sum_{i=1}^s m_i} \ge \sqrt{\frac{r+2}{r+3}} \sqrt{\frac{L^2}{r}}.$$

Next, suppose that $s=2$ and $(m_1,m_2) = (2,2)$. In this case, we use
an inequality that is stronger than \eqref{el}. With the notation as in
\eqref{el}, if $m_1 \ge 2$, then we have
\begin{eqnarray}\label{el-improved}
C^2 \ge
m_1^2+\ldots+m_s^2-m_1+\text{gon}(\tilde{C}).
\end{eqnarray}

Here $\tilde{C}$ is the normalization of $C$ and gon$(\tilde{C})$ is
the {\it gonality} of $\tilde{C}$ which is defined to be the least
degree of a covering $\tilde{C} \to \proj^1$. See \cite[Lemma
2.1]{Bas} and \cite[Theorem A]{KSS} for the single-point case and
\cite[Lemma 2.12]{Far} for the multi-point case. 

Since gon$(\tilde{C})$
is a positive integer, 
\eqref{el-improved} gives
$C^2 \ge 7$ in our situation. Now, using the Hodge Index Theorem as
above, we get $$\varepsilon(X,L,r) = \frac{L\cdot C}{4} \ge \sqrt{L^2}\sqrt{7/16} \ge \sqrt{\frac{r+2}{r+3}}\sqrt{\frac{L^2}{r}}.$$
The last inequality holds because 
$\frac{r+2}{r(r+3)} \le 7/16$ for $r \ge 2$.

Finally, let $s=1$ and $m=m_1$. 
By \eqref{el},
$C^2 \ge m^2-m$. Since $r \ge 2$, by hypothesis, we have
$\frac{r+2}{r(r+3)} \le 2/5$. Further, $m^2-m \ge 2m^2/5$ for
$m\ge 2$. So by the Hodge Index Theorem as above, we have
$$\varepsilon(X,L,r) = \frac{L\cdot C}{m} \ge \sqrt{L^2} \sqrt{2/5} \ge \sqrt{\frac{r+2}{r+3}} \sqrt{\frac{L^2}{r}}.$$

 For the last statement of the theorem, see Example
\ref{rational-ruled}. 
\end{proof}

\begin{example}\label{rational-ruled} In this example, we show that
  for every $r \ge 2$, there exists  a polarized
surface $(X,L)$ such that $\varepsilon(X,L,r) <
\sqrt{\frac{r+2}{r+3}}\sqrt{\frac{L^2}{r}}$.

For $r=2$, take $(X,L) = (\proj^2, \str_{\proj^2}(1))$. Then
$\varepsilon(X,L,2)=1/2 < \sqrt{4/5\cdot 1/2}$.

Let $r \ge 3$. Choose integers $n > e \ge  0$ such that
$r=2n-e+1$. 
Let $X$ be the ruled surface $\proj(\str_{\proj^1}\oplus
\str_{\proj^1}(-e))$ with a normalized section
$C_0$ and a fibre $f$. Let $\pi: X \to \proj^1$ be the canonical map. 
Let $L = C_0+nf$. Then $L$ is very ample by \cite[Chapter V, Theorem 2.17]{Har-AG}.

We have $L^2 = 2n-e=r-1$. 
Note that $K_X = -2C_0+(-2-e)f$ and
$H^1(X,L) = H^1(\proj^1,\pi_{\star}(L))=0$. 
Now it is easy to calculate, using Riemann-Roch, that
$h^0(X,L)=r+1$. Hence given any $r$ points $x_1,\ldots,x_r \in X$,
there is an effective divisor $D$ passing through $x_1,\ldots,x_r$
such that $D$ is linearly equivalent to $L$. 

So  $\varepsilon(X,L,r) \le \frac{r-1}{r} <
\sqrt{\frac{r+2}{r+3}}\sqrt{\frac{L^2}{r}}$. In fact, we claim that 
$\varepsilon(X,L,r) = \frac{r-1}{r}$. First note that, by Theorem \ref{main-general}, $\varepsilon(X,L,r)$
is computed by a curve $C$ passing through $s \le r$ very general points
with multiplicity one each. Then $C^2 \ge s-1$, by \eqref{el}. 
If $s=1$ then $\varepsilon(X,L,r) = L \cdot C \ge
1$, which is impossible since $\varepsilon(X,L,r) \le \frac{r-1}{r}
<1$. So let $s \ge 2$.
By the Hodge Index Theorem, $L\cdot C \ge \sqrt{L^2 C^2}\ge
\sqrt{(r-1)(s-1)}$.
Thus $$\varepsilon(X,L,r) = \frac{L\cdot C}{s} \ge \sqrt{\frac{(r-1)(s-1)}{s^2}} \ge
  \frac{r-1}{r}.$$

Note that the embedding of $X$ in
$\proj^r$ determined by $L$ is a rational normal scroll. This
example is also discussed in \cite[Example 4.2]{SS}.  In subsection 
\ref{rational-ruled-section}, we study Seshadri constants 
on rational ruled surfaces in more detail. 
\end{example}

\begin{remark}\label{compare}
We may compare our Theorem \ref{main-general} with \cite[Theorem
4.1]{SS}. This result says that if $\varepsilon(X,L,r) <
\sqrt{\frac{r-1}{r}}\sqrt{\frac{L^2}{r}}$, then $X$ has a fibration by
Seshadri curves. 
In situations where the existence of a Seshadri curve $C$ passing
through $s\le r$ points with multiplicity one each 
can be ruled out (see Remark \ref{condition1} for one such instance), the bound $\varepsilon(X,L,r) \ge
\sqrt{\frac{r+2}{r+3}}\sqrt{\frac{L^2}{r}}$ holds, by Theorem
\ref{main-general}. 
This is a better
bound than the one given in \cite[Theorem 4.1]{SS}.  
\end{remark}

\begin{remark}\label{condition1}
Suppose that for a polarized surface $(X,L)$ and an integer $r\ge 2$,
we have $\varepsilon(X,L,r) <
\sqrt{\frac{r+2}{r+3}}\sqrt{\frac{L^2}{r}}$. Then by Theorem
\ref{main-general}, $\varepsilon(X,L,r) = \frac{L\cdot C}{s}$ for a curve
$C$ passing through $s \le r$ very general points with multiplicity one
each.  Then $C$ must satisfy $C^2 < s$, as we show now. 
Suppose that $C^2 \ge s > 0$. 
By the Hodge Index Theorem,
$L\cdot C \ge \sqrt{L^2C^2}$. So
$\varepsilon(X,L,r) = \frac{L\cdot C}{s} \ge
\sqrt{\frac{L^2}{r}}\sqrt{\frac{C^2}{s}} \ge
\sqrt{\frac{L^2}{r}}$. So $\varepsilon(X,L,r)$ attains the maximal
possible value and this contradicts our assumption. 

In view of this remark, it is useful to investigate 
the following property:
If an irreducible and reduced curve $C$ on $X$
 passes through $r$ 
very general points, then $C^2 \ge r$.

If we have some information about curves $C$ with the above property
on a surface $X$, it may
be possible to establish that lower the bound in Theorem \ref{main-general}
holds. For instance, see Example \ref{k3}, Lemma \ref{rational-curve} and 
Theorem \ref{rational-ruled-seshadri}. 
\end{remark}

\begin{example} \label{k3}
Let $X$ be a K3 surface. Then  $h^1(\str_X) = 0$, $h^2(\str_X) =1$ and $K_X=\str_X$.
Let $C$ be an irreducible curve on $X$. Then $h^2(\str_X(C))=0$. 
Further, $h^1(\str_X(C))
=h^1(\str_X(-C))$, by Serre duality. Taking cohomology of the
exact sequence $0 \to \str_X(-C) \to \str_X \to \str_C \to 0,$ we
see that $h^1(\str_X(-C))=0$. For more details on K3 surfaces, see 
\cite[Chapter VIII]{BHPV} or \cite[Chapter VIII]{B}. 

So if $C$ is an irreducible curve on $X$, then $h^0(\str_X(C)) = \frac{C^2}{2}+2$, by Riemann-Roch. 
It is easy to see that $C$ passes through $r$ very general points if and
only if $h^0(\str_X(C)) \ge r+1$. Thus if $C$ passes through 
$r$ very general points, then $C^2 \ge 2r-2$. So if $r\ge 2$, it follows
that $C^2 \ge r$. 

Now let $L$ be an ample line bundle on $X$ and  let $r \ge L^2$. 
Suppose that there exists a Seshadri curve for $L$ 
passing through $s \le r$ very general points with multiplicity one each. 
If $s\ge 2$, then by the argument in the previous paragraph 
$C^2 \ge s$. So Remark \ref{condition1} shows that 
$\varepsilon(X,L,r) = \sqrt{\frac{L^2}{r}}$. If $s=1$, then 
$\varepsilon(X,L,r) = L \cdot C \ge 1$. On the other hand, 
$\sqrt{\frac{r+2}{r+3}}\sqrt{\frac{L^2}{r}} < 1$, if $r \ge L^2$.

Hence for an ample line bundle $L$ on a K3 surface $X$, we have
$\varepsilon(X,L,r) \ge \sqrt{\frac{r+2}{r+3}}\sqrt{\frac{L^2}{r}}$,
for $r \ge \text{max}\{L^2, 2\}$, by 
Theorem \ref{main-general} and Remark \ref{condition1}.
\end{example}

\section{Ruled surfaces}\label{ruled-surfaces}
Let $C$ be a smooth curve and let $\pi: X \to C$ be a ruled surface
over $C$. We choose a normalized vector bundle $E$ of rank 2 on $C$
such that $X \cong \mathbb{P}(E)$. Let $e = -{\rm deg}(E)$. 
Let $C_0$ be the image of a section of $\pi$ such that $C_0^2 = -e$
and let $f$ be a fibre of $\pi$. Then
the Picard group of $X$ modulo numerical equivalence is 
a free abelian group of rank 2 
generated by $C_o$ and $f$. We have $f^2=0$ and $C_0 \cdot f = 1$. 

A complete characterization of ample line bundles on $X$ is known. For
this, and other details on ruled surfaces, we refer to \cite[Chapter
V, Section
2]{Har-AG}. 

In this section, we consider a ruled surface $\pi: X \to C$ and an
ample line bundle $L$ on $X$. For an integer $r \ge 1$ and points
$x_1,\ldots,x_r \in X$, we are
interested in the problem of computing the 
Seshadri constants $\varepsilon(X,L,x_1,\ldots,x_r)$, or obtaining lower
bounds for the general Seshadri constant $\varepsilon(X,L,r)$. 
%In the first subsection, we make precise computations when $r \le e$. 

\subsection{The case $r\le e$} ~

Let $X \to C$ be a ruled surface with invariant $e \in
\ints$. In this subsection, we will assume that $1 \le r \le e$.  We
are primarily motivated by \cite[Theorems 4.1, 4.2]{Fue} and 
\cite[Theorem 3.27]{Syz} which compute the single-point
Seshadri constants on ruled surfaces. 
Our main result in this case is Theorem \ref{r<e}, 
which generalizes these results 
%in \cite[Theorems 4.1]{Fue} and  \cite[Theorem 3.27]{Syz}  
to the multi-point case when $e > 0$.

\begin{theorem}\label{r<e}
  Let $X$ be a ruled surface with invariant $e>0$. Fix a positive integer $r \le e$.
Let $L$ be an ample
line bundle on $X$ and let $x_1,\ldots,x_r \in X$.  Denote by $t$ the
maximum number
of points among $x_1,\ldots,x_r$ that lie on a single fibre $f$ and by $s$
the number of points among $x_1,\ldots,x_r$ lying on $C_0$. 

Then the Seshadri
constant $\varepsilon(X,L,x_1,\ldots,x_r) = {\rm min}\left\{\frac{L\cdot
      f}{t},\frac{L\cdot C_o}{s}\right\}$, if $s > 0$. Otherwise, 
$\varepsilon(X,L,x_1,\ldots,x_r) =\frac{L\cdot f}{t}$.
\end{theorem}
\begin{proof}
Let $L = aC_0+bf$. By \cite[Chapter V, Proposition 2.20]{Har-AG}, $a>0$ and $b
> ae$. Let $D=\alpha C_0 + \beta f$ be any irreducible and reduced
curve on $X$ such that
$D \ne C_0$ and $D\ne f$. 
In this case, we have $\alpha > 0$ and $\beta \ge \alpha e$, by the
same reference as above.  

Suppose  that $D$ passes through 
$x_1,\ldots, x_r$ with multiplicities $m_1,\ldots,m_r$. Assume that
$m := \sum_{i=1}^r m_i > 0$.  
We will show that $\frac{L\cdot D}{m} \ge \frac{L\cdot f}{t}$
and the theorem will follow.

Since $D\ne f$, $\alpha = D\cdot f \ge m_i$ for all $i$. This follows
by considering a fibre through the point $x_i$. Hence $r\alpha \ge
m$. Note that $b \ge ae+1$ and $\beta \ge \alpha e$. So 
%Write $b = ae+1+n$ where $n\ge 0$. 

$$\frac{L\cdot D}{m} = \frac{-a\alpha e +\alpha b + a\beta}{m} 
%= \frac{\alpha+\alpha n  + a\beta}{m} 
\ge \frac{\alpha +
  a\alpha e}{m} \ge \frac{\alpha(1+ae)}{m} \ge \frac{ae+1}{r}.$$

Since $e \ge r$,  $\frac{L\cdot D}{m}  \ge \frac{ae+1}{r} \ge
\frac{a}{t} = \frac{L\cdot f}{t}$.
\end{proof}
\begin{corollary}\label{cor-r<e}
With $X, e, x_1,\ldots,x_r,t$ and $s$ as in Theorem \ref{r<e}, let $L
= aC_0+bf$ be an ample line bundle
on $X$ such that $b \ge 2ae+1$. Then 
$\varepsilon(X,L,x_1,\ldots,x_r) =\frac{L\cdot f}{t}$.
\end{corollary}
\begin{proof}
Assume that $s >0$, as otherwise there is nothing to prove. 
We will prove that $\frac{L\cdot C_0}{s} \ge \frac{L\cdot f}{t}$.  
We saw above in the proof of Theorem \ref{r<e} that $ \frac{ae+1}{r} \ge \frac{L\cdot
  f}{t}$. So it suffices to show that 
$\frac{ae+1}{r} \le \frac{L\cdot C_0}{s}$. 

If not, we have $\frac{ae+1}{r}  > \frac{L\cdot C_0}{s}$, which
implies that $s >\left(\frac{b-ae}{ae+1}\right)r \ge r$, since $b \ge
2ae+1$. But this is absurd, since there are only $r$ points. 
\end{proof}

Given a positive rational number $q$, we show next that there exists
a surface $X$ and an ample line bundle $L$ on $X$ such that
for $r$ sufficiently large, there are points $x_1,\ldots,x_r\in X$ with
$\varepsilon(X,L,x_1,\ldots,x_r)=q$.
In particular, this shows that multi-point Seshadri constants for
arbitrarily large $r$ can be
arbitrarily small. Compare this with Miranda's example \cite[Example
5.2.1]{L} which shows
that single-point Seshadri constants can be arbitrarily small.

\begin{corollary}\label{any-q}
Let $q = \frac{a}{t}$ be a rational number with $a,t > 0$. Then there
exist a polarized ruled surface $(X,L)$, an integer $r$ and points
$x_1,\ldots,x_r \in X$ such that $\varepsilon(X,L,x_1,\ldots,x_r) = q$.
\end{corollary}
\begin{proof}
Let $r=e=t$. Consider a
ruled surface $X$ with invariant $e$. For specificity, we may consider
the rational ruled surface $\proj(\str_{\proj^1}\oplus
\str_{\proj^1}(-e))$. Let $L = aC_0+(2ae+1)f$. Then $L$ is ample. 
Choose points $x_1,\ldots,x_r$ on $X$ so that the maximum number of
points lying on a single fibre is $t$.
Then, by Corollary \ref{cor-r<e}, $\varepsilon(X,L,x_1,\ldots,x_r) = a/t$.
\end{proof}

Note that in the above corollary, $L^2 = 3a^2e+2a$, so that
$\sqrt{\frac{L^2}{r}} \ge a \ge 1$. Thus the smallness of the Seshadri
  constant $\varepsilon(X,L,x_1,\ldots,x_r)$ is {\it not} due to the smallness of
  $\sqrt{\frac{L^2}{r}}$. Contrast this with the case of $(\proj^2,\str_{\proj^2}(1))$. 

While the multi-point Seshadri constant $\varepsilon(X,L,x_1,\ldots,x_r)$ can be arbitrarily small at
special points $x_1,\ldots,x_r$, its value at very general points is
always $L\cdot f$, as we show now. 

\begin{corollary}\label{general-r<e}
Let $X$ be a ruled surface with invariant $e$ and let $1 \le r \le e$. For
an ample line bundle $L$, one has $\varepsilon(X,L,r) = L\cdot f$. 
\end{corollary}
\begin{proof}
Note that $\varepsilon(X,L,r)$ is attained at $r$ very general points
$x_1,\ldots,x_r \in X$. Since $C_0^2 = -e < 0$, no very general point lies
on $C_0$. Let $f$ be the fibre through $x_1$.  Since the points
$x_1,\ldots,x_r$ are very general, $x_i \notin f$ for $i=2,\ldots,r$.
Thus in the notation of Theorem \ref{r<e}, $s=0$ and
$t=1$. Hence $\varepsilon(X,L,r) = \varepsilon(X,L,x_1,\ldots,x_r)= L \cdot f$. 
\end{proof}

\begin{remark}
Let $X$ be a ruled surface with invariant $e$ and let $1 \le r \le
e$. Then it follows from Corollary \ref{general-r<e} that the Seshadri
constant $\varepsilon(X,L,r)$ is sub-maximal for {\it any} ample line
bundle $L$ on $X$. Indeed, let $L = aC_0+bf$. Since $L$ is ample,
$b>ae$. So $L^2 = -a^2e+2ab = a(2b-ae) > a^2e$. So $\frac{L^2}{r} >
\frac{a^2e}{r} \ge a^2$. Hence $a = \varepsilon(X,L,r) < \sqrt{L^2/r}$. 

For any $r$, if $e\gg r$, then $\varepsilon(X,L,r)=a$ would be very small
compared to $\sqrt{\frac{L^2}{r}} = \sqrt{\frac{2ab-a^2e}{r}}$;
in particular, smaller than
$\sqrt{\frac{r-1}{r}}\sqrt{\frac{L^2}{r}}$. By \cite[Theorem 4.1]{SS},
it follows that $X$ is fibred by Seshadri curves.  Of course, in this
particular case this is obvious, since the fibres $f$ are Seshadri
curves. 

\end{remark}

\subsection{Rational ruled surfaces} \label{rational-ruled-section}

In this subsection, we will consider rational ruled surfaces and
clarify the bound given in Theorem \ref{main-general}. 

Let $X$ be a ruled surface over $\proj^1$. We fix a normalized vector
bundle $E = \str_{\proj^1} \oplus \str_{\proj^1}(-e)$ such that
$X = \proj(E)$ and $e \ge 0$. Denote by $C_0$ a section with $C_0^2 = -e$ and
let $f$ be a fibre.

First, we make the following observation. 

\begin{lemma}\label{rational-curve}
Let $X$ be a rational ruled surface. Let $C$ be an irreducible curve
on $X$ passing through $r\ge 1$ very general points. If $C^2 < r$, then
$C^2=r-1$ and $C$ is a smooth rational curve. 
\end{lemma}
\begin{proof}
By \eqref{el}, we have $C^2 \ge r-1$, in any case. So $C^2 = r-1$.
We also note that $C$ passes through $r$ very general
points if any only if $h^0(X,\str_X(C)) \ge r+1$. This follows from a
simple dimension count on the linear system $|C|$ and the fact the
very general points impose independent conditions. 
%We note that the irregularity $h^1(X,\str_X)= 0$, so that algebraic 
%equivalence is the same as linear equivalence. 
In our situation, note also that
$h^0(X,\str_X(C))=r+1$, because if $h^0(X,\str_X(C)) > r+1$, then $C$
passes through more than $r$ very general points and $C^2 \ge r$ by \eqref{el}.

%Since $h^1(X,\str_X(C)) = h^2(X,\str_X(C)) = 0$, and $h^0(X,\str_X(C)=r+1$, 
Since $C$ is effective, $h^2(X,\str_X(C)) = 0$. Using the projection
formula to push-down $\str_X(C)$ to $\proj^1$, we see that $h^1(X,\str_X(C)) = 0$.
Now since $h^0(X,\str_X(C))=r+1$, the Riemann-Roch theorem gives 
$K_X \cdot C = -r-1$ and  
by adjunction, $2p_a(C)-2 = C^2+K_X\cdot C = -2$. Hence $p_a(C)=0$
and $C$ is a smooth rational curve. 
\end{proof}

\begin{lemma}\label{list-rational-curves}
Let $X$ be a rational ruled surface with invariant $e \ge 0$,
normalized section $C_0$ and a fibre $f$. 
Let $C=mC_0+nf$ be an irreducible smooth curve on $X$. Then $C$ is rational 
if and only if 
\begin{enumerate}
\item $C  = C_0$, or 
\item $C=f$, or 
\item $m=1, n > e$, or
\item $e > 0, m=1, n=e$, or  
\item $e=0,  m \ge 1, n=1$, or 
\item $e=1, m=n=2$.
\end{enumerate}
\end{lemma}
\begin{proof}
In each of the cases listed above, it is easy to see, using the adjunction formula,
that $C$ is rational. 

For the converse, let $C$ be an
irreducible, smooth 
rational curve and suppose that $C \ne C_0, C\ne f$.
By \cite[Chapter V, Corollary 2.8]{Har-AG}, 
$m > 0, n > me$ or $e > 0, m > 0, n=me$. 

We have $C^2 = -m^2e+2mn$ and $K_X \cdot C = me-2n-2m$. 
By adjunction, 
\begin{eqnarray}\label{adjunction}
-2 = C^2+K_X\cdot C = -m^2e+me-2m+n(2m-2).
\end{eqnarray}

{\bf Case 1:} $n > me$. 

So  $-2 = -m^2e+me-2m+n(2m-2) \ge -m^2e+me-2m+(me+1)(2m-2) =
m^2e-me-2=me(m-1)-2$. So $0 \ge me(m-1)$. Since 
$m \ge 1, e\ge 0$, we have $m=1$ or $e=0$. If $m=1$, then we have case (3).

Let $e=0$.  Then $-2=2(mn-m-n)$, by \eqref{adjunction}. So either
$m=1$ or $n=1$.  If $m=1$, we have case (3). If $n=1$, we have case
(5). 

{\bf Case 2:} $n=me$. 

In this case, $e> 0$. By \eqref{adjunction}, $-2 = (m^2-m)e-2m$. Hence
$2m-2 = (m^2-m)e \ge m^2-m$. Hence $m=1$ or $m=2, e=1$. 
If $m=1$, we have case (4) and if $m=2, e=1$, we have case (6).
\end{proof}

Now we state our main theorem on rational ruled surfaces. 

\begin{theorem}\label{rational-ruled-seshadri}
Let $X$ be a rational ruled surface. Let $L$ be an ample
line bundle on $X$. 
Then $$\varepsilon(X,L,r) \ge
\sqrt{\frac{r+2}{r+3}}\sqrt{\frac{L^2}{r}}, \text{~ for ~} r \ge L^2+5.$$
\end{theorem}
\begin{proof}
%If the theorem is false, then we have 
%$\varepsilon(X,L,r) < 
%\sqrt{\frac{r+2}{r+3}}\sqrt{\frac{L^2}{r}}$, and obtain a
%contradiction. 

If the statement of the theorem is false, then by Theorem
\ref{main-general}, 
there is an irreducible and reduced curve 
$C$ passing through $s \le r$ very general points with multiplicity one 
each such that $\varepsilon(X,L,r) = \frac{L\cdot C}{s}$. We will show
this is impossible. 

By Remark 
\ref{condition1}, $C$ must satisfy $C^2=s-1$. By Lemma
\ref{rational-curve}, $C$ is a smooth rational curve. 
We will consider four cases below
which deal with all the six possibilities
listed in Lemma \ref{list-rational-curves}.
%{\bf Case 1} deals with (1), (2); {\bf Case 2} deals with 
%(3), (4); {\bf Case 3} deals with (5) and {\bf Case 4} deals with
%(6). 

Let $L = aC_0+bf$ with $a > 0$ and $b > ae$. 

{\bf Case 1:} 
Let $C = C_0$. Then $C^2 = -e \le 0$. On the other hand, $C^2
=s-1\ge 0$. So $e=0$ and $s=1$. Then $\varepsilon(X,L,r) = L \cdot C_0
\ge 1$. But this is a contradiction because we have $r \ge L^2$
and $\varepsilon(X,L,r) < \sqrt{\frac{r+2}{r+3}}\sqrt{\frac{L^2}{r}} <
1$.  
The same argument holds when $C=f$.

{\bf Case 2:}
Let $C = C_0 + nf$ for $n \ge e$. 

Then $L\cdot C = an+b-ae$ and
$C^2 = 2n-e=s-1$. So $\varepsilon(X,L,r) = \frac{L \cdot C}{s} =
\frac{an+b-ae}{2n-e+1} \ge \frac{an+1}{2n-e+1}$. The last inequality 
holds because $b>ae$.  Thus if $a \ge 2$,  $\varepsilon(X,L,r) = \frac{L
  \cdot C}{s} \ge 1$, again contradicting our assumption. 

Now let $a=1$. The desired contradiction follows from 
the inequality:
\begin{eqnarray}\label{to-show}
\frac{(L\cdot C)^2}{s^2} \ge
\left(\frac{r+2}{r+3}\right)\frac{L^2}{r}. 
\end{eqnarray}

We will now establish \eqref{to-show}.
If $s=1$,  then $\varepsilon(X,L,r) \ge 1$. So there is nothing to
prove. 

Let $2 \le s < r$. 
By the Hodge Index Theorem, 
\eqref{to-show} follows if we show that $\frac{C^2}{s^2} \ge \frac{r+2}{r(r+3)}$. 
This, in turn, is equivalent to $r(r+3)(s-1) \ge (r+2)s^2$. It is not
difficult to check this inequality holds when 
$2 \le s \le  r-1$ and $r \ge 4$. Indeed, writing the difference of
the two terms as a quadratic in $s$, we have
$Q(s) = -(r+2)s^2+r(r+3)s-r(r+3)$. Its graph is a downward sloping 
parabola and it is easy to see that $Q(2) \ge 0$ and $Q(r-1) \ge 0$
when $r\ge 4$. 

Now let $s=r$.
By hypothesis, $4 \le r-1-L^2 = C^2-L^2 = 2(n-b)$. Hence 
$n-b\ge 2$. Note that $L^2 = 2b-e$, $C^2 = 2n-e=r-1$ and 
$L\cdot C  =b+n-e$. Thus  
$(L\cdot C)^2 - L^2 C^2 = (n-b)^2$. We also
have $C^2 = L^2 + 2(n-b)$ and $L^2 = C^2-2(n-b)=r-1-2(n-b)$.

\begin{eqnarray*}
\eqref{to-show} &\Leftrightarrow&\frac{(L\cdot C)^2}{r^2} \ge \left(\frac{r+2}{r+3}\right)\frac{L^2}{r}\\
&\Leftrightarrow& L^2C^2+(n-b)^2 \ge \left(\frac{r(r+2)}{r+3}\right)L^2\\
&\Leftrightarrow& (L^2)^2+2(n-b)L^2+(n-b)^2  \ge \left(\frac{r(r+2)}{r+3}\right)L^2\\
&\Leftrightarrow& L^2 + 2(n-b)+\frac{(n-b)^2}{L^2} \ge  \frac{r(r+2)}{r+3}\\
&\Leftrightarrow& r-1 + \frac{(n-b)^2}{L^2} \ge  \frac{r(r+2)}{r+3}\\
&\Leftrightarrow& \frac{(n-b)^2(r+3)}{L^2} \ge  3.
\end{eqnarray*} 

The last inequality holds because $n-b \ge 2$ and $r \ge L^2+5$. 

{\bf Case 3:}
Let $e=0, m\ge 1$ and $n=1$. 

Then $C^2 = 2m$ and $s = 2m+1$. So $\frac{L\cdot C}{s} =
\frac{a+bm}{2m+1}$. If $b \ge 2$, then 
$\varepsilon(X,L,r) = \frac{L\cdot C}{s} \ge 1$, which is a
contradiction. Hence we have $b=1$ and $L = aC_0+f$. The argument here
is very similar to the argument in {\bf Case 2} with $e = 0$. 

{\bf Case 4:} Finally, let $C = 2C_0+2f$ and $e=1$. We will show that 
$C$ can not be a Seshadri curve in this case. 

We have $C^2 = 4$, $s  = 5$.  Now $L \cdot C = 2b$ and 
$\frac{L\cdot C}{s} = \frac{2b}{5}$. If $a \ge 2$, then $b \ge 3$. So
$\varepsilon(X,L,r) = \frac{L\cdot C}{s} \ge 1$, which is a
contradiction because $r \ge L^2$. So $a=1$ and $b=2$. Thus $\frac{L\cdot C}{s} = 4/5$. On
the other hand, $\sqrt{\frac{L^2}{r}}\sqrt{\frac{r+2}{r+3}} =
\sqrt{\frac{3(r+2)}{r(r+3)}}$. It is easy to see that 
$4/5 \ge  \sqrt{\frac{3(r+2)}{r(r+3)}}$ for all $r \ge 5$. 

This completes the proof of the theorem. 
\end{proof}

\bibliographystyle{plain}

\end{document}